\documentclass[a4j,reqno]{amsart}
\usepackage{amsmath}
\usepackage{amssymb}
\usepackage{amsthm}
\usepackage{amscd}
\usepackage{stmaryrd}
\usepackage[dvipdfmx]{graphicx}
\usepackage[all]{xy}
\usepackage{wrapfig}
\usepackage{color}
\usepackage{latexsym}
\usepackage{enumerate}
\usepackage{cases}
\usepackage[top=30mm,bottom=30mm,left=30mm,right=30mm]{geometry}
\usepackage{graphics}
\numberwithin{equation}{section}
\theoremstyle{definition}
\newtheorem{example}{Example}[section]

\newtheorem{lemma}[example]{Lemma}
\newtheorem{theorem}[example]{Theorem}
\newtheorem{proposition}[example]{Proposition}

\DeclareMathOperator{\add}{\mathsf{add}}
\DeclareMathOperator{\proj}{\mathsf{proj}}
\DeclareMathOperator{\inj}{\mathsf{inj}}
\DeclareMathOperator{\Fac}{\mathsf{Fac}}
\DeclareMathOperator{\Sub}{\mathsf{Sub}}
\DeclareMathOperator{\Mod}{\mathsf{mod}}
\DeclareMathOperator{\thick}{\mathsf{thick}}

\title{Wide subcategories are semistable}
\author{toshiya yurikusa}
\address{T. Yurikusa: Graduate School of Mathematics, Nagoya University, Chikusa-ku, Nagoya, 464-8602 Japan}
\email{m15049q@math.nagoya-u.ac.jp}

\begin{document}


\maketitle
\begin{abstract}
 For an arbitrary finite dimensional algebra $\Lambda$, we prove that any wide subcategory of $\Mod \Lambda$ satisfying a certain finiteness condition is $\theta$-semistable for some stability condition $\theta$. More generally, we show that wide subcategories of $\Mod \Lambda$ associated with two-term presilting complexes of $\Lambda$ are semistable. This provides a complement for Ingalls-Thomas-type bijections for finite dimensional algebras.
\end{abstract}


\section{Introduction}

 This classification problem of subcategories is a well studied subject in representation theory, algebraic geometry and algebraic topology (e.g. \cite{Hop,N,Th}). Among others, we refer to \cite{B,Hov,IT,KS,MS,Ta} for recent developments on the classification of {\it wide subcategories}, which are full subcategories of an abelian category closed under kernels, cokernels and extensions.

 Important examples of wide subcategories are given by geometric invariant theory for quiver representations \cite{K}. Recall that a stability condition on $\Mod \Lambda$ for a finite dimensional algebra $\Lambda$ is a linear form $\theta$ on $K_0(\Mod \Lambda) \otimes_{\mathbb Z} {\mathbb R}$, where $K_0(\Mod \Lambda)$ is the Grothendieck group of $\Mod \Lambda$. We say that $M \in \Mod \Lambda$ is $\theta$-{\it semistable} if $\theta(M)=0$ and $\theta(L) \le 0$ for any submodule $L$ of $M$, or equivalently, $\theta(N) \ge 0$ for any factor module $N$ of $M$. The full subcategory of $\theta$-semistable $\Lambda$-modules is called the $\theta$-{\it semistable subcategory} of $\Mod \Lambda$. It is basic that semistable subcategories of $\Mod \Lambda$ are wide.

 For quiver representations, Ingalls and Thomas \cite{IT} gave bijections between wide/semistable subcategories and other important objects: For the path algebra $kQ$ of a finite connected acyclic quiver $Q$ over a field $k$, there are bijections (called {\it Ingalls-Thomas bijections}) between the following objects, where we refer to Subsection \ref{prelim} for unexplained terminologies.\par
\begin{enumerate}
 \item Isomorphism classes of basic support tilting modules in $\Mod (kQ)$.
 \item Functorially finite torsion classes in $\Mod (kQ)$.
 \item Functorially finite wide subcategories of $\Mod (kQ)$.
 \item Functorially finite semistable subcategories of $\Mod (kQ)$.
\end{enumerate}
 They also proved that (1)-(4) above correspond bijectively with the clusters in the cluster algebra of $Q$ and the isomorphism classes of basic cluster tilting objects in the cluster category of $kQ$.


 Later, works of Adachi-Iyama-Reiten \cite{AIR} and Marks-Stovicek \cite{MS} gave the following Ingalls-Thomas-type bijections for an arbitrary finite dimensional $k$-algebra, where we refer to Subsection \ref{prelim} for unexplained terminologies and explicit bijections.

\begin{theorem}\label{MSresult}\cite[Theorem 0.5]{AIR}\cite[Theorem 3.10]{MS}
 Let $\Lambda$ be a finite dimensional algebra over a field $k$. There are bijections between the following objects:\par
\begin{enumerate}
 \item Isomorphism classes of basic support $\tau$-tilting modules in $\Mod \Lambda$.
 \item[$(1')$] Isomorphism classes of basic two-term silting complexes in ${\mathsf K}^{{\rm b}}(\proj \Lambda)$.
 \item Functorially finite torsion classes in $\Mod \Lambda$.
 \item[$(2')$] Functorially finite torsion free classes in $\Mod \Lambda$.
 \item Left finite wide subcategories of $\Mod \Lambda$.
\end{enumerate}
\end{theorem}

 Notice that the statement for semistable subcategories of $\Mod \Lambda$ is missing in Theorem \ref{MSresult}. The aim of this paper is to prove the following complement of Theorem \ref{MSresult}.


\begin{theorem}\label{main1}
 For a finite dimensional algebra $\Lambda$ over a field $k$, the following objects are the same.\par
\begin{enumerate}
 \item[(3)] Left finite wide subcategories of $\Mod \Lambda$.
 \item[(4)] Left finite semistable subcategories of $\Mod \Lambda$.
\end{enumerate}
 Therefore, there are bijections between (1)-(4) in Theorem \ref{MSresult}.
\end{theorem}


 To construct a stability condition $\theta$ for a given left finite wide subcategory, we need the following preparation. Let $T$ be a basic two-term silting complex in ${\mathsf K}^{{\rm b}}(\proj \Lambda)$. Then there is a decomposition $T=T_{\lambda} \oplus T_{\rho}$ and a triangle
\begin{equation}\label{triangle}
 \Lambda \rightarrow T' \rightarrow T'' \rightarrow \Lambda[1]
\end{equation}
in ${\mathsf K}^{{\rm b}}(\proj \Lambda)$, where $\add T'=\add T_{\lambda}$ and $\add T''=\add T_{\rho}$ (see \cite[Proposition 2.24]{AI}). Then $T$ corresponds to the left finite wide subcategory
\begin{equation}\label{W^T}
 \mathcal{W}^{T}:=\Fac {\rm H}^0(T) \cap {\rm H}^0(T_{\rho})^{\perp}
\end{equation}
via the bijection between $(1')$ and (3) in Theorem \ref{MSresult} (see Subsection \ref{prelim}).

 Our Theorem \ref{main1} is a consequence of the following result, where $\langle -,-\rangle$ is the Euler form (see (\ref{Euler})).

\begin{theorem}\label{main}
 Let $\Lambda$ be a finite dimensional algebra over a field $k$. Let $T$ be a basic two-term silting complex in ${\mathsf K}^{{\rm b}}(\proj \Lambda)$. We consider an ${\mathbb R}$-linear form $\theta$ defined by
\[
\sum_{X} a_{X} \langle X,-\rangle : K_0(\Mod \Lambda) \otimes_{\mathbb Z} {\mathbb R} \rightarrow {\mathbb R},
\]
where $X$ runs over all indecomposable direct summands of $T_{\rho}$, and $a_{X}$ is an arbitrary positive real number for each $X$. Then ${\mathcal W}^T$ is the $\theta$-semistable subcategory of $\Mod \Lambda$.
\end{theorem}


 We prove Theorem \ref{main} in a more general setting. Any basic two-term presilting complex $U$ in ${\mathsf K}^{{\rm b}}(\proj \Lambda)$ gives rise to a wide subcategory of $\Mod \Lambda$ as follows: By \cite[Proposition 2.9]{AIR} (see also \cite[Section 5]{BPP}), there are two torsion pairs
{\setlength\arraycolsep{0.5mm}
\begin{eqnarray*}
 (\mathcal{T}_U^{+},\mathcal{F}_U^{+})&:=&({}^{\perp}{\rm H}^{-1}(\nu U),\Sub {\rm H}^{-1}(\nu U)),\\
 (\mathcal{T}_U^{-},\mathcal{F}_U^{-})&:=&(\Fac {\rm H}^0(U),{\rm H}^0(U)^{\perp})
\end{eqnarray*}}
in $\Mod \Lambda$ such that $\mathcal{T}_U^{+} \supseteq \mathcal{T}_U^{-}$ and $\mathcal{F}_U^{+} \subseteq \mathcal{F}_U^{-}$. Then
\[
 \mathcal{W}_U:=\mathcal{T}_U^{+}\cap\mathcal{F}_U^{-}
\]
is a wide subcategory of $\Mod \Lambda$ (e.g. \cite{DIRRT}), which is equivalent to $\Mod C$ for some explicitly constructed finite dimensional algebra $C$ (see \cite[Theorem 1.4]{J}).

 Our Theorem \ref{main} can be deduced from the following result since $\mathcal{W}^{T}=\mathcal{W}_{T_{\rho}}$ holds for any two-term silting complex $T$ (see Lemma \ref{2wide}).

\begin{theorem}\label{presiltsemistable}
 Let $U$ be a basic two-term presilting complex in ${\mathsf K}^{{\rm b}}(\proj \Lambda)$. We consider an ${\mathbb R}$-linear form $\theta$ defined by
\[
\sum_{X} a_{X} \langle X,-\rangle : K_0(\Mod \Lambda) \otimes_{\mathbb Z} {\mathbb R} \rightarrow {\mathbb R},
\]
where $X$ runs over all indecomposable direct summands of $U$, and $a_{X}$ is an arbitrary positive real number for each $X$. Then $\mathcal{W}_U$ is the $\theta$-semistable subcategory of $\Mod \Lambda$.
\end{theorem}

 Note that in the context of support $\tau$-tilting modules, Theorem \ref{presiltsemistable} was independently obtained by Br\"{u}stle-Smith-Treffinger \cite{BST} and Speyer-Thomas \cite{ST}.


\medskip\noindent{\bf Notations}.
 Let $\Lambda$ be a finite dimensional algebra over a field $k$, and $\Mod \Lambda$ (resp., $\proj \Lambda$, $\inj \Lambda$) the category of finitely generated right $\Lambda$-modules (resp., projective right $\Lambda$-modules, injective right $\Lambda$-modules). For $M \in \Mod \Lambda$, let $\add M$ (resp., $\Fac M$, $\Sub M$) be the category of all direct summands (resp., factor modules, submodules) of finite direct sums of copies of $M$. We denote by $D$ the $k$-dual ${\rm Hom}_{k}(-,k)$.

 For a full subcategory ${\mathcal S}$ of $\Mod \Lambda$, let
\[
{\mathcal S}^{\perp}:=\{M \in \Mod \Lambda \mid {\rm Hom}_{\Lambda}({\mathcal S},M)=0\},\hspace{3mm}
{}^{\perp}{\mathcal S}:=\{M \in \Mod \Lambda \mid {\rm Hom}_{\Lambda}(M,{\mathcal S})=0\}.
\]

 For an additive (resp., abelian) category ${\mathcal A}$, let ${\mathsf K}^{{\rm b}}({\mathcal A})$ (resp., ${\mathsf D}^{{\rm b}}({\mathcal A})$) be the homotopy (resp., derived) category of bounded complexes over ${\mathcal A}$. We denote by $\nu$ the Nakayama functor $D \Lambda \otimes_{\Lambda} - : {\mathsf K}^{{\rm b}}(\proj \Lambda) \rightarrow {\mathsf K}^{{\rm b}}(\inj \Lambda)$.


\section{Example}

 Before proving our results, we give an example.

 Let $Q$ be the quiver
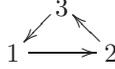
\[
\begin{xy}
 (0,-2)="1"*{1}, +(13,0)="2"*{2}, +(-6.5,6)="3"*{3}
 \ar@{->}"1"+(2,0);"2"+(-2,0) \ar@{->}"2"+(-1.5,1.5);"3"+(1.5,-1) \ar@{->}"3"+(-1.5,-1);"1"+(1.5,1.5)
\end{xy}
\]
and $I$ be the two-sided ideal of the path algebra $kQ$ generated by all paths of length three. Then $\Lambda := kQ/I$ is a finite dimensional $k$-algebra. The Auslander-Reiten quiver of $\Mod \Lambda$ is the following
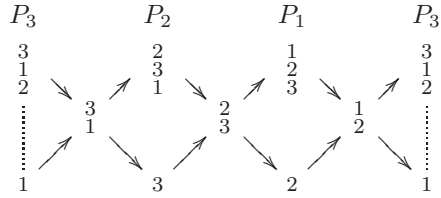
\[
\xymatrix@!C=2mm@!R=0.5mm{
 \hspace{-1mm}\mbox{\footnotesize$\renewcommand{\arraystretch}{0.6}\begin{array}{c}\mbox{{\normalsize$P_3$}}\\ \\ 3\\ 1\\ 2\\ \end{array}$}\hspace{-1mm} \ar[rd] && \hspace{-1mm}\mbox{\footnotesize$\renewcommand{\arraystretch}{0.6}\begin{array}{c}\mbox{{\normalsize$P_2$}}\\ \\ 2\\ 3\\ 1\\ \end{array}$}\hspace{-1mm} \ar[rd] && \hspace{-1mm}\mbox{\footnotesize$\renewcommand{\arraystretch}{0.6}\begin{array}{c}\mbox{{\normalsize$P_1$}}\\ \\ 1\\ 2\\ 3\\ \end{array}$}\hspace{-1mm} \ar[rd] && \hspace{-1mm}\mbox{\footnotesize$\renewcommand{\arraystretch}{0.6}\begin{array}{c}\mbox{{\normalsize$P_3$}}\\ \\ 3\\ 1\\ 2\\ \end{array}$}\hspace{-1mm}\\
 & \hspace{-1mm}\mbox{\footnotesize$\renewcommand{\arraystretch}{0.6}\begin{array}{c} 3\\ 1\\ \end{array}$}\hspace{-1mm} \ar[ru] \ar[rd] && \hspace{-1mm}\mbox{\footnotesize$\renewcommand{\arraystretch}{0.6}\begin{array}{c} 2\\ 3\\ \end{array}$}\hspace{-1mm} \ar[ru] \ar[rd] && \hspace{-1mm}\mbox{\footnotesize$\renewcommand{\arraystretch}{0.6}\begin{array}{c} 1\\ 2\\ \end{array}$}\hspace{-1mm} \ar[ru] \ar[rd] &\\
 \hspace{-1mm}\mbox{\footnotesize$\renewcommand{\arraystretch}{0.6}\begin{array}{c} 1\\ \end{array}$}\hspace{-1mm} \ar[ru] \ar@{.}[uu] && \hspace{-1mm}\mbox{\footnotesize$\renewcommand{\arraystretch}{0.6}\begin{array}{c} 3\\ \end{array}$}\hspace{-1mm} \ar[ru] && \hspace{-1mm}\mbox{\footnotesize$\renewcommand{\arraystretch}{0.6}\begin{array}{c} 2\\ \end{array}$}\hspace{-1mm} \ar[ru] &&
\mbox{\footnotesize$\renewcommand{\arraystretch}{0.6}\begin{array}{c} 1\\ \end{array}$} \ar@{.}[uu]}
\]
 Table \ref{taubijex} gives a complete list of two-term silting complexes, support $\tau$-tiling $\Lambda$-modules, functorially finite torsion classes and left finite wide subcategories in $\Mod \Lambda$. The objects in each row correspond to each other under the bijections of Theorem \ref{MSresult}. For $T \in {\rm 2 \mathchar`-silt}\Lambda$, we write the class of indecomposable direct summands of $T$ in $K_0(\proj \Lambda)$. Moreover, indecomposable direct summands $X$ of $T_{\rho}$ and ${\rm H}^0 (X)$ are colored in blue.


\begin{table}[htbp]
  \caption{Example of Theorem \ref{MSresult}}
\begin{itemize}\setlength{\itemsep}{1mm}
  \item ${\rm 2 \mathchar`-silt}\Lambda$ : \mbox{the set of isomorphism classes of basic two-term silting complexes in } ${\mathsf K}^{{\rm b}}(\proj \Lambda)$.
  \item ${\rm s \tau \mathchar`-tilt}\Lambda : \mbox{the set of isomorphism classes of basic support $\tau$-tilting modules in } \Mod \Lambda.$
  \item ${\rm f \mathchar`- tors}\Lambda : \mbox{the set of functorially finite torsion classes in } \Mod \Lambda.$
  \item ${\rm f_{L} \mathchar`-wide}\Lambda : \mbox{the set of left finite wide subcategories of } \Mod \Lambda.$
\end{itemize}
\vspace{3mm}
\begin{center}
\begin{minipage}{0.48\hsize}
  \begin{tabular}{c|c|c|c}
    ${\rm 2 \mathchar`-silt}\Lambda$ & ${\rm s \tau \mathchar`-tilt}\Lambda$ & ${\rm f \mathchar`- tors}\Lambda$ & ${\rm f_{L} \mathchar`-wide}\Lambda$ \\[1pt] \hline \hline
    \begin{xy}
     (0,-3.5)*{P_1,P_2,P_3}
    \end{xy} &
    \begin{xy}
     (0,0)="0"*{{\color{red} \bullet}}, +(4,0)*{\color{red} \bullet}, +(4,0)*{\color{red} \bullet},
     "0"+(2,-3.5)*{\circ}, +(4,0)*{\circ}, +(4,0)*{\circ},
     "0"+(0,-7)*{\circ}, +(4,0)*{\circ}, +(4,0)*{\circ}, +(0,-2.3)*{}
    \end{xy} &
    \begin{xy}
     (0,0)="0"*{{\color{red} \bullet}}, +(4,0)*{\color{red} \bullet}, +(4,0)*{\color{red} \bullet},
     "0"+(2,-3.5)*{{\color{red} \bullet}}, +(4,0)*{{\color{red} \bullet}}, +(4,0)*{{\color{red} \bullet}},
     "0"+(0,-7)*{{\color{red} \bullet}}, +(4,0)*{{\color{red} \bullet}}, +(4,0)*{{\color{red} \bullet}}, +(0,-2.3)*{}
    \end{xy} &
    \begin{xy}
     (0,0)="0"*{{\color{red} \bullet}}, +(4,0)*{\color{red} \bullet}, +(4,0)*{\color{red} \bullet},
     "0"+(2,-3.5)*{{\color{red} \bullet}}, +(4,0)*{{\color{red} \bullet}}, +(4,0)*{{\color{red} \bullet}},
     "0"+(0,-7)*{{\color{red} \bullet}}, +(4,0)*{{\color{red} \bullet}}, +(4,0)*{{\color{red} \bullet}}, +(0,-2.3)*{}
    \end{xy}
\\ \hline
    \begin{xy}
     (0,-1)*{P_1,P_2,}, +(0,-5)*{{\color{blue}P_2-P_3}}
    \end{xy} &
    \begin{xy}
     (0,0)="0"*{\circ}, +(4,0)*{\color{red} \bullet}, +(4,0)*{\color{red} \bullet},
     "0"+(2,-3.5)*{\circ}, +(4,0)*{\circ}, +(4,0)*{\circ},
     "0"+(0,-7)*{\circ}, +(4,0)*{\circ}, +(4,0)*{{\color{blue} \bullet}}, +(0,-2.3)*{}
    \end{xy} &
    \begin{xy}
     (0,0)="0"*{\circ}, +(4,0)*{\color{red} \bullet}, +(4,0)*{\color{red} \bullet},
     "0"+(2,-3.5)*{\circ}, +(4,0)*{{\color{red} \bullet}}, +(4,0)*{{\color{red} \bullet}},
     "0"+(0,-7)*{{\color{red} \bullet}}, +(4,0)*{\circ}, +(4,0)*{{\color{red} \bullet}}, +(0,-2.3)*{}
    \end{xy} &
    \begin{xy}
     (0,0)="0"*{\circ}, +(4,0)*{\color{red} \bullet}, +(4,0)*{\color{red} \bullet},
     "0"+(2,-3.5)*{\circ}, +(4,0)*{{\color{red} \bullet}}, +(4,0)*{\circ},
     "0"+(0,-7)*{{\color{red} \bullet}}, +(4,0)*{\circ}, +(4,0)*{\circ}, +(0,-2.3)*{}
    \end{xy}
\\ \hline
    \begin{xy}
     (0,-1)*{P_1,P_3,}, +(0,-5)*{{\color{blue}P_1-P_2}}
    \end{xy} &
    \begin{xy}
     (0,0)="0"*{{\color{red} \bullet}}, +(4,0)*{\circ}, +(4,0)*{\color{red} \bullet},
     "0"+(2,-3.5)*{\circ}, +(4,0)*{\circ}, +(4,0)*{\circ},
     "0"+(0,-7)*{\color{blue} \bullet}, +(4,0)*{\circ}, +(4,0)*{\circ}, +(0,-2.3)*{}
    \end{xy} &
    \begin{xy}
     (0,0)="0"*{{\color{red} \bullet}}, +(4,0)*{\circ}, +(4,0)*{\color{red} \bullet},
     "0"+(2,-3.5)*{{\color{red} \bullet}}, +(4,0)*{\circ}, +(4,0)*{{\color{red} \bullet}},
     "0"+(0,-7)*{{\color{red} \bullet}}, +(4,0)*{{\color{red} \bullet}}, +(4,0)*{\circ}, +(0,-2.3)*{}
    \end{xy} &
    \begin{xy}
     (0,0)="0"*{{\color{red} \bullet}}, +(4,0)*{\circ}, +(4,0)*{\color{red} \bullet},
     "0"+(2,-3.5)*{\circ}, +(4,0)*{\circ}, +(4,0)*{{\color{red} \bullet}},
     "0"+(0,-7)*{\circ}, +(4,0)*{{\color{red} \bullet}}, +(4,0)*{\circ}, +(0,-2.3)*{}
    \end{xy}
\\ \hline
    \begin{xy}
     (0,-1)*{P_2,P_3,}, +(0,-5)*{{\color{blue}P_3-P_1}}
    \end{xy} &
    \begin{xy}
     (0,0)="0"*{{\color{red} \bullet}}, +(4,0)*{\color{red} \bullet}, +(4,0)*{\circ},
     "0"+(2,-3.5)*{\circ}, +(4,0)*{\circ}, +(4,0)*{\circ},
     "0"+(0,-7)*{\circ}, +(4,0)*{\color{blue} \bullet}, +(4,0)*{\circ}, +(0,-2.3)*{}
    \end{xy} &
    \begin{xy}
     (0,0)="0"*{{\color{red} \bullet}}, +(4,0)*{\color{red} \bullet}, +(4,0)*{\circ},
     "0"+(2,-3.5)*{{\color{red} \bullet}}, +(4,0)*{{\color{red} \bullet}}, +(4,0)*{\circ},
     "0"+(0,-7)*{\circ}, +(4,0)*{{\color{red} \bullet}}, +(4,0)*{{\color{red} \bullet}}, +(0,-2.3)*{}
    \end{xy} &
    \begin{xy}
     (0,0)="0"*{{\color{red} \bullet}}, +(4,0)*{\color{red} \bullet}, +(4,0)*{\circ},
     "0"+(2,-3.5)*{{\color{red} \bullet}}, +(4,0)*{\circ}, +(4,0)*{\circ},
     "0"+(0,-7)*{\circ}, +(4,0)*{\circ}, +(4,0)*{{\color{red} \bullet}}, +(0,-2.3)*{}
    \end{xy}
\\ \hline
    \begin{xy}
     (0,-1)*{P_1,{\color{blue}P_1-P_3},}, +(0,-5)*{P_2-P_3}
    \end{xy} &
    \begin{xy}
     (0,0)="0"*{\circ}, +(4,0)*{\circ}, +(4,0)*{\color{red} \bullet},
     "0"+(2,-3.5)*{\circ}, +(4,0)*{\circ}, +(4,0)*{{\color{blue} \bullet}},
     "0"+(0,-7)*{\circ}, +(4,0)*{\circ}, +(4,0)*{\color{red} \bullet}, +(0,-2.3)*{}
    \end{xy} &
    \begin{xy}
     (0,0)="0"*{\circ}, +(4,0)*{\circ}, +(4,0)*{\color{red} \bullet},
     "0"+(2,-3.5)*{\circ}, +(4,0)*{\circ}, +(4,0)*{{\color{red} \bullet}},
     "0"+(0,-7)*{{\color{red} \bullet}}, +(4,0)*{\circ}, +(4,0)*{{\color{red} \bullet}}, +(0,-2.3)*{}
    \end{xy} &
    \begin{xy}
     (0,0)="0"*{\circ}, +(4,0)*{\circ}, +(4,0)*{\color{red} \bullet},
     "0"+(2,-3.5)*{\circ}, +(4,0)*{\circ}, +(4,0)*{\circ},
     "0"+(0,-7)*{\circ}, +(4,0)*{\circ}, +(4,0)*{\color{red} \bullet}, +(0,-2.3)*{}
    \end{xy}
\\ \hline
    \begin{xy}
     (0,-1)*{P_2,{\color{blue}P_2-P_1},}, +(0,-5)*{{\color{blue}P_2-P_3}}
    \end{xy} &
    \begin{xy}
     (0,0)="0"*{\circ}, +(4,0)*{\color{red} \bullet}, +(4,0)*{\circ},
     "0"+(2,-3.5)*{\circ}, +(4,0)*{{\color{blue} \bullet}}, +(4,0)*{\circ},
     "0"+(0,-7)*{\circ}, +(4,0)*{\circ}, +(4,0)*{\color{blue} \bullet}, +(0,-2.3)*{}
    \end{xy} &
    \begin{xy}
     (0,0)="0"*{\circ}, +(4,0)*{\color{red} \bullet}, +(4,0)*{\circ},
     "0"+(2,-3.5)*{\circ}, +(4,0)*{{\color{red} \bullet}}, +(4,0)*{\circ},
     "0"+(0,-7)*{\circ}, +(4,0)*{\circ}, +(4,0)*{\color{red} \bullet}, +(0,-2.3)*{}
    \end{xy} &
    \begin{xy}
     (0,0)="0"*{\circ}, +(4,0)*{\color{red} \bullet}, +(4,0)*{\circ},
     "0"+(2,-3.5)*{\circ}, +(4,0)*{\circ}, +(4,0)*{\circ},
     "0"+(0,-7)*{\circ}, +(4,0)*{\circ}, +(4,0)*{\circ}, +(0,-2.3)*{}
    \end{xy}
\\ \hline
    \begin{xy}
     (0,-1)*{P_3,{\color{blue}P_3-P_2},}, +(0,-5)*{P_1-P_2}
    \end{xy} &
    \begin{xy}
     (0,0)="0"*{{\color{red} \bullet}}, +(4,0)*{\circ}, +(4,0)*{\circ},
     "0"+(2,-3.5)*{\color{blue} \bullet}, +(4,0)*{\circ}, +(4,0)*{\circ},
     "0"+(0,-7)*{\color{red} \bullet}, +(4,0)*{\circ}, +(4,0)*{\circ}, +(0,-2.3)*{}
    \end{xy} &
    \begin{xy}
     (0,0)="0"*{{\color{red} \bullet}}, +(4,0)*{\circ}, +(4,0)*{\circ},
     "0"+(2,-3.5)*{\color{red} \bullet}, +(4,0)*{\circ}, +(4,0)*{\circ},
     "0"+(0,-7)*{\color{red} \bullet}, +(4,0)*{{\color{red} \bullet}}, +(4,0)*{\circ}, +(0,-2.3)*{}
    \end{xy} &
    \begin{xy}
     (0,0)="0"*{{\color{red} \bullet}}, +(4,0)*{\circ}, +(4,0)*{\circ},
     "0"+(2,-3.5)*{\circ}, +(4,0)*{\circ}, +(4,0)*{\circ},
     "0"+(0,-7)*{\color{red} \bullet}, +(4,0)*{\circ}, +(4,0)*{\circ}, +(0,-2.3)*{}
    \end{xy}
\\ \hline
    \begin{xy}
     (0,-1)*{P_1,{\color{blue}P_1-P_3},}, +(0,-5)*{{\color{blue}P_1-P_2}}
    \end{xy} &
    \begin{xy}
     (0,0)="0"*{\circ}, +(4,0)*{\circ}, +(4,0)*{\color{red} \bullet},
     "0"+(2,-3.5)*{\circ}, +(4,0)*{\circ}, +(4,0)*{\color{blue} \bullet},
     "0"+(0,-7)*{{\color{blue} \bullet}}, +(4,0)*{\circ}, +(4,0)*{\circ}, +(0,-2.3)*{}
    \end{xy} &
    \begin{xy}
     (0,0)="0"*{\circ}, +(4,0)*{\circ}, +(4,0)*{\color{red} \bullet},
     "0"+(2,-3.5)*{\circ}, +(4,0)*{\circ}, +(4,0)*{\color{red} \bullet},
     "0"+(0,-7)*{{\color{red} \bullet}}, +(4,0)*{\circ}, +(4,0)*{\circ}, +(0,-2.3)*{}
    \end{xy} &
    \begin{xy}
     (0,0)="0"*{\circ}, +(4,0)*{\circ}, +(4,0)*{\color{red} \bullet},
     "0"+(2,-3.5)*{\circ}, +(4,0)*{\circ}, +(4,0)*{\circ},
     "0"+(0,-7)*{\circ}, +(4,0)*{\circ}, +(4,0)*{\circ}, +(0,-2.3)*{}
    \end{xy}
\\ \hline
    \begin{xy}
     (0,-1)*{P_2,{\color{blue}P_2-P_1},}, +(0,-5)*{P_3-P_1}
    \end{xy} &
    \begin{xy}
     (0,0)="0"*{\circ}, +(4,0)*{\color{red} \bullet}, +(4,0)*{\circ},
     "0"+(2,-3.5)*{\circ}, +(4,0)*{{\color{blue} \bullet}}, +(4,0)*{\circ},
     "0"+(0,-7)*{\circ}, +(4,0)*{\color{red} \bullet}, +(4,0)*{\circ}, +(0,-2.3)*{}
    \end{xy} &
    \begin{xy}
     (0,0)="0"*{\circ}, +(4,0)*{\color{red} \bullet}, +(4,0)*{\circ},
     "0"+(2,-3.5)*{\circ}, +(4,0)*{{\color{red} \bullet}}, +(4,0)*{\circ},
     "0"+(0,-7)*{\circ}, +(4,0)*{\color{red} \bullet}, +(4,0)*{\color{red} \bullet}, +(0,-2.3)*{}
    \end{xy} &
    \begin{xy}
     (0,0)="0"*{\circ}, +(4,0)*{\color{red} \bullet}, +(4,0)*{\circ},
     "0"+(2,-3.5)*{\circ}, +(4,0)*{\circ}, +(4,0)*{\circ},
     "0"+(0,-7)*{\circ}, +(4,0)*{\color{red} \bullet}, +(4,0)*{\circ}, +(0,-2.3)*{}
    \end{xy}
\\ \hline
    \begin{xy}
     (0,-1)*{P_3,{\color{blue}P_3-P_2},}, +(0,-5)*{{\color{blue}P_3-P_1}}
    \end{xy} &
    \begin{xy}
     (0,0)="0"*{{\color{red} \bullet}}, +(4,0)*{\circ}, +(4,0)*{\circ},
     "0"+(2,-3.5)*{\color{blue} \bullet}, +(4,0)*{\circ}, +(4,0)*{\circ},
     "0"+(0,-7)*{\circ}, +(4,0)*{\color{blue} \bullet}, +(4,0)*{\circ}, +(0,-2.3)*{}
    \end{xy} &
    \begin{xy}
     (0,0)="0"*{{\color{red} \bullet}}, +(4,0)*{\circ}, +(4,0)*{\circ},
     "0"+(2,-3.5)*{\color{red} \bullet}, +(4,0)*{\circ}, +(4,0)*{\circ},
     "0"+(0,-7)*{\circ}, +(4,0)*{{\color{red} \bullet}}, +(4,0)*{\circ}, +(0,-2.3)*{}
    \end{xy} &
    \begin{xy}
     (0,0)="0"*{{\color{red} \bullet}}, +(4,0)*{\circ}, +(4,0)*{\circ},
     "0"+(2,-3.5)*{\circ}, +(4,0)*{\circ}, +(4,0)*{\circ},
     "0"+(0,-7)*{\circ}, +(4,0)*{\circ}, +(4,0)*{\circ}, +(0,-2.3)*{}
    \end{xy}
  \end{tabular}
\end{minipage}
\begin{minipage}{0.5\hsize}
  \begin{tabular}{c|c|c|c}
    ${\rm 2 \mathchar`-silt}\Lambda$ & ${\rm s \tau \mathchar`-tilt}\Lambda$ & ${\rm f \mathchar`- tors}\Lambda$ & ${\rm f_{L} \mathchar`-wide}\Lambda$ \\[1pt] \hline \hline
    \begin{xy}
     (0,-1)*{P_1-P_3,}, +(0,-5)*{P_2-P_3,{\color{blue}-P_3}}
    \end{xy} &
    \begin{xy}
     (0,0)="0"*{\circ}, +(4,0)*{\circ}, +(4,0)*{\circ},
     "0"+(2,-3.5)*{\circ}, +(4,0)*{\circ}, +(4,0)*{{\color{red} \bullet}},
     "0"+(0,-7)*{\circ}, +(4,0)*{\circ}, +(4,0)*{\color{red} \bullet}, +(0,-2.3)*{}
    \end{xy} &
    \begin{xy}
     (0,0)="0"*{\circ}, +(4,0)*{\circ}, +(4,0)*{\circ},
     "0"+(2,-3.5)*{\circ}, +(4,0)*{\circ}, +(4,0)*{{\color{red} \bullet}},
     "0"+(0,-7)*{{\color{red} \bullet}}, +(4,0)*{\circ}, +(4,0)*{{\color{red} \bullet}}, +(0,-2.3)*{}
    \end{xy} &
    \begin{xy}
     (0,0)="0"*{\circ}, +(4,0)*{\circ}, +(4,0)*{\circ},
     "0"+(2,-3.5)*{\circ}, +(4,0)*{\circ}, +(4,0)*{{\color{red} \bullet}},
     "0"+(0,-7)*{{\color{red} \bullet}}, +(4,0)*{\circ}, +(4,0)*{{\color{red} \bullet}}, +(0,-2.3)*{}
    \end{xy}
\\ \hline
    \begin{xy}
     (0,-1)*{P_2-P_1,}, +(0,-5)*{{\color{blue}P_2-P_3},{\color{blue}-P_1}}
    \end{xy} &
    \begin{xy}
     (0,0)="0"*{\circ}, +(4,0)*{\circ}, +(4,0)*{\circ},
     "0"+(2,-3.5)*{\circ}, +(4,0)*{{\color{red} \bullet}}, +(4,0)*{\circ},
     "0"+(0,-7)*{\circ}, +(4,0)*{\circ}, +(4,0)*{\color{blue} \bullet}, +(0,-2.3)*{}
    \end{xy} &
    \begin{xy}
     (0,0)="0"*{\circ}, +(4,0)*{\circ}, +(4,0)*{\circ},
     "0"+(2,-3.5)*{\circ}, +(4,0)*{{\color{red} \bullet}}, +(4,0)*{\circ},
     "0"+(0,-7)*{\circ}, +(4,0)*{\circ}, +(4,0)*{\color{red} \bullet}, +(0,-2.3)*{}
    \end{xy} &
    \begin{xy}
     (0,0)="0"*{\circ}, +(4,0)*{\circ}, +(4,0)*{\circ},
     "0"+(2,-3.5)*{\circ}, +(4,0)*{{\color{red} \bullet}}, +(4,0)*{\circ},
     "0"+(0,-7)*{\circ}, +(4,0)*{\circ}, +(4,0)*{\circ}, +(0,-2.3)*{}
    \end{xy}
\\ \hline
    \begin{xy}
     (0,-1)*{P_3-P_2,}, +(0,-5)*{P_1-P_2,{\color{blue}-P_2}}
    \end{xy} &
    \begin{xy}
     (0,0)="0"*{\circ}, +(4,0)*{\circ}, +(4,0)*{\circ},
     "0"+(2,-3.5)*{\color{red} \bullet}, +(4,0)*{\circ}, +(4,0)*{\circ},
     "0"+(0,-7)*{\color{red} \bullet}, +(4,0)*{\circ}, +(4,0)*{\circ}, +(0,-2.3)*{}
    \end{xy} &
    \begin{xy}
     (0,0)="0"*{\circ}, +(4,0)*{\circ}, +(4,0)*{\circ},
     "0"+(2,-3.5)*{\color{red} \bullet}, +(4,0)*{\circ}, +(4,0)*{\circ},
     "0"+(0,-7)*{\color{red} \bullet}, +(4,0)*{{\color{red} \bullet}}, +(4,0)*{\circ}, +(0,-2.3)*{}
    \end{xy} &
    \begin{xy}
     (0,0)="0"*{\circ}, +(4,0)*{\circ}, +(4,0)*{\circ},
     "0"+(2,-3.5)*{\color{red} \bullet}, +(4,0)*{\circ}, +(4,0)*{\circ},
     "0"+(0,-7)*{\color{red} \bullet}, +(4,0)*{{\color{red} \bullet}}, +(4,0)*{\circ}, +(0,-2.3)*{}
    \end{xy}
\\ \hline
    \begin{xy}
     (0,-1)*{P_1-P_3,}, +(0,-5)*{{\color{blue}P_1-P_2},{\color{blue}-P_3}}
    \end{xy} &
    \begin{xy}
     (0,0)="0"*{\circ}, +(4,0)*{\circ}, +(4,0)*{\circ},
     "0"+(2,-3.5)*{\circ}, +(4,0)*{\circ}, +(4,0)*{\color{red} \bullet},
     "0"+(0,-7)*{{\color{blue} \bullet}}, +(4,0)*{\circ}, +(4,0)*{\circ}, +(0,-2.3)*{}
    \end{xy} &
    \begin{xy}
     (0,0)="0"*{\circ}, +(4,0)*{\circ}, +(4,0)*{\circ},
     "0"+(2,-3.5)*{\circ}, +(4,0)*{\circ}, +(4,0)*{\color{red} \bullet},
     "0"+(0,-7)*{{\color{red} \bullet}}, +(4,0)*{\circ}, +(4,0)*{\circ}, +(0,-2.3)*{}
    \end{xy} &
    \begin{xy}
     (0,0)="0"*{\circ}, +(4,0)*{\circ}, +(4,0)*{\circ},
     "0"+(2,-3.5)*{\circ}, +(4,0)*{\circ}, +(4,0)*{\color{red} \bullet},
     "0"+(0,-7)*{\circ}, +(4,0)*{\circ}, +(4,0)*{\circ}, +(0,-2.3)*{}
    \end{xy}
\\ \hline
    \begin{xy}
     (0,-1)*{P_2-P_1,}, +(0,-5)*{P_3-P_1,{\color{blue}-P_1}}
    \end{xy} &
    \begin{xy}
     (0,0)="0"*{\circ}, +(4,0)*{\circ}, +(4,0)*{\circ},
     "0"+(2,-3.5)*{\circ}, +(4,0)*{{\color{red} \bullet}}, +(4,0)*{\circ},
     "0"+(0,-7)*{\circ}, +(4,0)*{\color{red} \bullet}, +(4,0)*{\circ}, +(0,-2.3)*{}
    \end{xy} &
    \begin{xy}
     (0,0)="0"*{\circ}, +(4,0)*{\circ}, +(4,0)*{\circ},
     "0"+(2,-3.5)*{\circ}, +(4,0)*{{\color{red} \bullet}}, +(4,0)*{\circ},
     "0"+(0,-7)*{\circ}, +(4,0)*{\color{red} \bullet}, +(4,0)*{\color{red} \bullet}, +(0,-2.3)*{}
    \end{xy} &
    \begin{xy}
     (0,0)="0"*{\circ}, +(4,0)*{\circ}, +(4,0)*{\circ},
     "0"+(2,-3.5)*{\circ}, +(4,0)*{{\color{red} \bullet}}, +(4,0)*{\circ},
     "0"+(0,-7)*{\circ}, +(4,0)*{\color{red} \bullet}, +(4,0)*{\color{red} \bullet}, +(0,-2.3)*{}
    \end{xy}
\\ \hline
    \begin{xy}
     (0,-1)*{P_3-P_2,}, +(0,-5)*{{\color{blue}P_3-P_1},{\color{blue}-P_2}}
    \end{xy} &
    \begin{xy}
     (0,0)="0"*{\circ}, +(4,0)*{\circ}, +(4,0)*{\circ},
     "0"+(2,-3.5)*{\color{red} \bullet}, +(4,0)*{\circ}, +(4,0)*{\circ},
     "0"+(0,-7)*{\circ}, +(4,0)*{\color{blue} \bullet}, +(4,0)*{\circ}, +(0,-2.3)*{}
    \end{xy} &
    \begin{xy}
     (0,0)="0"*{\circ}, +(4,0)*{\circ}, +(4,0)*{\circ},
     "0"+(2,-3.5)*{\color{red} \bullet}, +(4,0)*{\circ}, +(4,0)*{\circ},
     "0"+(0,-7)*{\circ}, +(4,0)*{{\color{red} \bullet}}, +(4,0)*{\circ}, +(0,-2.3)*{}
    \end{xy} &
    \begin{xy}
     (0,0)="0"*{\circ}, +(4,0)*{\circ}, +(4,0)*{\circ},
     "0"+(2,-3.5)*{\color{red} \bullet}, +(4,0)*{\circ}, +(4,0)*{\circ},
     "0"+(0,-7)*{\circ}, +(4,0)*{\circ}, +(4,0)*{\circ}, +(0,-2.3)*{}
    \end{xy}
\\ \hline
    \begin{xy}
     (0,-1)*{P_2-P_3,}, +(0,-5)*{{\color{blue}-P_1},{\color{blue}-P_3}}
    \end{xy} &
    \begin{xy}
     (0,0)="0"*{\circ}, +(4,0)*{\circ}, +(4,0)*{\circ},
     "0"+(2,-3.5)*{\circ}, +(4,0)*{\circ}, +(4,0)*{\circ},
     "0"+(0,-7)*{\circ}, +(4,0)*{\circ}, +(4,0)*{\color{red} \bullet}, +(0,-2.3)*{}
    \end{xy} &
    \begin{xy}
     (0,0)="0"*{\circ}, +(4,0)*{\circ}, +(4,0)*{\circ},
     "0"+(2,-3.5)*{\circ}, +(4,0)*{\circ}, +(4,0)*{\circ},
     "0"+(0,-7)*{\circ}, +(4,0)*{\circ}, +(4,0)*{\color{red} \bullet}, +(0,-2.3)*{}
    \end{xy} &
    \begin{xy}
     (0,0)="0"*{\circ}, +(4,0)*{\circ}, +(4,0)*{\circ},
     "0"+(2,-3.5)*{\circ}, +(4,0)*{\circ}, +(4,0)*{\circ},
     "0"+(0,-7)*{\circ}, +(4,0)*{\circ}, +(4,0)*{\color{red} \bullet}, +(0,-2.3)*{}
    \end{xy}
\\ \hline
    \begin{xy}
     (0,-1)*{P_1-P_2,}, +(0,-5)*{{\color{blue}-P_2},{\color{blue}-P_3}}
    \end{xy} &
    \begin{xy}
     (0,0)="0"*{\circ}, +(4,0)*{\circ}, +(4,0)*{\circ},
     "0"+(2,-3.5)*{\circ}, +(4,0)*{\circ}, +(4,0)*{\circ},
     "0"+(0,-7)*{\color{red} \bullet}, +(4,0)*{\circ}, +(4,0)*{\circ}, +(0,-2.3)*{}
    \end{xy} &
    \begin{xy}
     (0,0)="0"*{\circ}, +(4,0)*{\circ}, +(4,0)*{\circ},
     "0"+(2,-3.5)*{\circ}, +(4,0)*{\circ}, +(4,0)*{\circ},
     "0"+(0,-7)*{\color{red} \bullet}, +(4,0)*{\circ}, +(4,0)*{\circ}, +(0,-2.3)*{}
    \end{xy} &
    \begin{xy}
     (0,0)="0"*{\circ}, +(4,0)*{\circ}, +(4,0)*{\circ},
     "0"+(2,-3.5)*{\circ}, +(4,0)*{\circ}, +(4,0)*{\circ},
     "0"+(0,-7)*{\color{red} \bullet}, +(4,0)*{\circ}, +(4,0)*{\circ}, +(0,-2.3)*{}
    \end{xy}
\\ \hline
    \begin{xy}
     (0,-1)*{P_3-P_1,}, +(0,-5)*{{\color{blue}-P_1},{\color{blue}-P_2}}
    \end{xy} &
    \begin{xy}
     (0,0)="0"*{\circ}, +(4,0)*{\circ}, +(4,0)*{\circ},
     "0"+(2,-3.5)*{\circ}, +(4,0)*{\circ}, +(4,0)*{\circ},
     "0"+(0,-7)*{\circ}, +(4,0)*{\color{red} \bullet}, +(4,0)*{\circ}, +(0,-2.3)*{}
    \end{xy} &
    \begin{xy}
     (0,0)="0"*{\circ}, +(4,0)*{\circ}, +(4,0)*{\circ},
     "0"+(2,-3.5)*{\circ}, +(4,0)*{\circ}, +(4,0)*{\circ},
     "0"+(0,-7)*{\circ}, +(4,0)*{\color{red} \bullet}, +(4,0)*{\circ}, +(0,-2.3)*{}
    \end{xy} &
    \begin{xy}
     (0,0)="0"*{\circ}, +(4,0)*{\circ}, +(4,0)*{\circ},
     "0"+(2,-3.5)*{\circ}, +(4,0)*{\circ}, +(4,0)*{\circ},
     "0"+(0,-7)*{\circ}, +(4,0)*{\color{red} \bullet}, +(4,0)*{\circ}, +(0,-2.3)*{}
    \end{xy}
\\ \hline
    \begin{xy}
     (0,-1)*{{\color{blue}-P_1},}, +(0,-5)*{{\color{blue}-P_2},{\color{blue}-P_3}}
    \end{xy} &
    \begin{xy}
     (0,0)="0"*{\circ}, +(4,0)*{\circ}, +(4,0)*{\circ},
     "0"+(2,-3.5)*{\circ}, +(4,0)*{\circ}, +(4,0)*{\circ},
     "0"+(0,-7)*{\circ}, +(4,0)*{\circ}, +(4,0)*{\circ}, +(0,-2.3)*{}
    \end{xy} &
    \begin{xy}
     (0,0)="0"*{\circ}, +(4,0)*{\circ}, +(4,0)*{\circ},
     "0"+(2,-3.5)*{\circ}, +(4,0)*{\circ}, +(4,0)*{\circ},
     "0"+(0,-7)*{\circ}, +(4,0)*{\circ}, +(4,0)*{\circ}, +(0,-2.3)*{}
    \end{xy} &
    \begin{xy}
     (0,0)="0"*{\circ}, +(4,0)*{\circ}, +(4,0)*{\circ},
     "0"+(2,-3.5)*{\circ}, +(4,0)*{\circ}, +(4,0)*{\circ},
     "0"+(0,-7)*{\circ}, +(4,0)*{\circ}, +(4,0)*{\circ}, +(0,-2.3)*{}
    \end{xy}
  \end{tabular}
\end{minipage}
\end{center}\vspace{10mm}
  \label{taubijex}
\end{table}


 For a basic two-term presilting complex $U=U_1 \oplus \ldots \oplus U_m$ with indecomposable direct summands $U_i$ in ${\mathsf K}^{{\rm b}}(\proj \Lambda)$, we consider the cone
\[
 C(U):=\Bigl\{\sum_{i=1}^{m}a_i [U_i] \mid a_i>0\ \ (1 \le i \le m)\Bigr\} \subseteq K_0(\proj \Lambda)\otimes_{\mathbb Z} {\mathbb R}.
\]
 Since $\Lambda$ is $\tau$-tilting finite, we have a decomposition \cite{DIJ}
\[
 K_0(\proj \Lambda)\otimes_{\mathbb Z} {\mathbb R}=\underset{U}{\bigsqcup}\hspace{1mm}C(U),
\]
where $U$ runs over isomorphism classes of basic two-term presilting complexes in ${\mathsf K}^{{\rm b}}(\proj \Lambda)$ (see Figure \ref{chamber}). By Theorem \ref{presiltsemistable}, any $\theta$ in the cone $C(U)$ gives rise to the wide/semistable subcategory $\mathcal{W}_U$ of $\Mod \Lambda$.

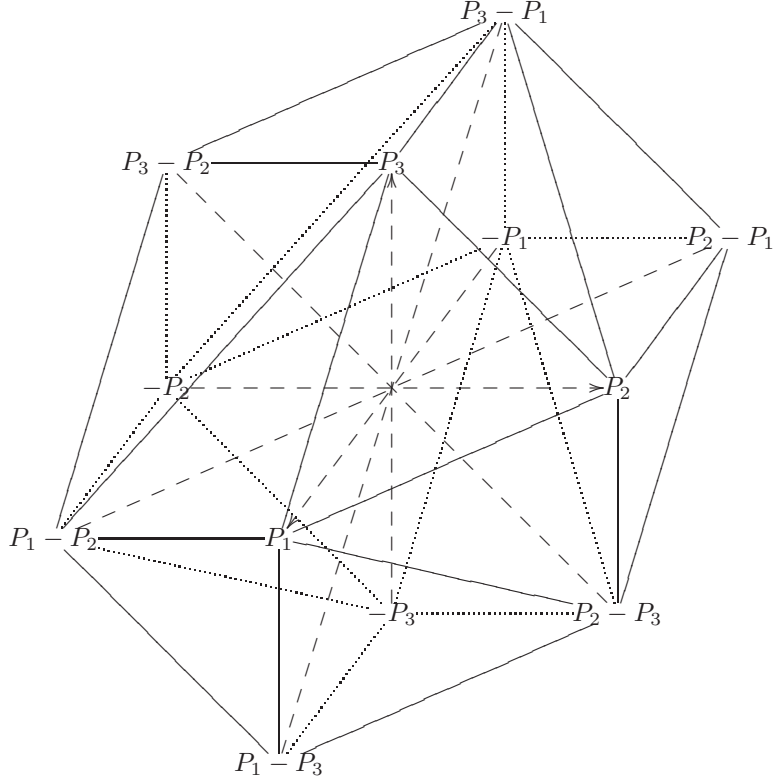
\begin{figure}[htbp]
  \caption{The decomposition of $K_0(\proj \Lambda)\otimes_{\mathbb Z} {\mathbb R}$}\label{chamber}
\[
\begin{xy}
   (0,0)="0", +(-15,-20)*{P_1}="1", +(30,40)*{-P_1}="-1", "0"+(30,0)*{P_2}="2", +(-60,0)*{-P_2}="-2", "0"+(0,30)*{P_3}="3", +(0,-60)*{-P_3}="-3", "1"+(30,0)="12", +(-60,0)*{P_1-P_2}="1-2", "-1"+(30,0)*{P_2-P_1}="2-1", +(-60,0)="-1-2", "1"+(0,30)="13", +(0,-60)*{P_1-P_3}="1-3", "-1"+(0,30)*{P_3-P_1}="3-1", +(0,-60)="-1-3", "2"+(0,30)="23", +(0,-60)*{P_2-P_3}="2-3", "-2"+(0,30)*{P_3-P_2}="3-2", +(0,-60)="-2-3"
   \ar@{-->}"-1";"1" \ar@{-->}"-2";"2" \ar@{-->}"-3";"3" \ar@{--}"1-2";"2-1" \ar@{--}"1-3";"3-1" \ar@{--}"2-3";"3-2"
   \ar@{-}"1";"2" \ar@{-}"2";"3" \ar@{-}"3";"1" \ar@{-}"1";"2-3" \ar@{-}"2";"2-3" \ar@{-}"1";"1-3" \ar@{-}"2-3";"1-3" \ar@{-}"1";"1-2" \ar@{-}"3";"1-2" \ar@{-}"3";"3-2" \ar@{-}"1-2";"3-2" \ar@{-}"1-2";"1-3" \ar@{-}"3";"3-1" \ar@{-}"2";"3-1" \ar@{-}"2-1";"3-1" \ar@{-}"2";"2-1" \ar@{-}"2-1";"2-3" \ar@{-}"3-1";"3-2"
   \ar@{.}"-3";"2-3" \ar@{.}"-3";"1-3" \ar@{.}"-3";"1-2" \ar@{.}"-2";"-3" \ar@{.}"-2";"1-2" \ar@{.}"-2";"3-2" \ar@{.}"-1";"3-1" \ar@{.}"-1";"-2" \ar@{.}"-2";"3-1" \ar@{.}"-3";"-1" \ar@{.}"-1";"2-1" \ar@{.}"-1";"2-3"
\end{xy}
\]
\end{figure}


 (1) We consider the case $U=U_1 \oplus U_2$, where
\[
   U_1=(P_3\rightarrow P_2),\hspace{2mm}
   U_2=(P_1\rightarrow 0).
\]
 Then the cone $C(U)=\{a_{1}(P_2-P_3)+a_{2}(-P_1) \mid a_{1}>0, a_{2}>0\}$ gives rise to the wide/semistable subcategory
\[
 \mathcal{T}_U^{+}\hspace{3mm}
    \begin{xy}
     (0,3)="0"*{\circ}, +(4,0)*{\circ}, +(4,0)*{\circ},
     "0"+(2,-3.5)*{\circ}, +(4,0)*{\color{red} \bullet}, +(4,0)*{\circ},
     "0"+(0,-7)*{\circ}, +(4,0)*{\circ}, +(4,0)*{\color{red} \bullet}, +(0,-2.3)*{}
    \end{xy}
\hspace{2mm} \cap \hspace{2mm} \mathcal{F}_U^{-}\hspace{3mm}
    \begin{xy}
     (0,3)="0"*{\circ}, +(4,0)*{\color{red} \bullet}, +(4,0)*{\color{red} \bullet},
     "0"+(2,-3.5)*{\color{red} \bullet}, +(4,0)*{\color{red} \bullet}, +(4,0)*{\circ},
     "0"+(0,-7)*{\color{red} \bullet}, +(4,0)*{\color{red} \bullet}, +(4,0)*{\circ}, +(0,-2.3)*{}
    \end{xy}
\hspace{2mm} = \hspace{2mm} \mathcal{W}_U\hspace{3mm}
    \begin{xy}
     (0,3)="0"*{\circ}, +(4,0)*{\circ}, +(4,0)*{\circ},
     "0"+(2,-3.5)*{\circ}, +(4,0)*{\color{red} \bullet}, +(4,0)*{\circ},
     "0"+(0,-7)*{\circ}, +(4,0)*{\circ}, +(4,0)*{\circ}, +(0,-2.3)*{}
    \end{xy}
\hspace{2mm} = \hspace{2mm}
\add \Bigl(\hspace{-1mm}\mbox{\footnotesize$\renewcommand{\arraystretch}{0.6}\begin{array}{c} 2\\ 3\\ \end{array}$}\hspace{-1mm}\Bigr)
\]
 of $\Mod \Lambda$. On the other hand, let
\[
\theta=a_{U_1}\langle U_1,-\rangle + a_{U_2}\langle U_2,-\rangle : K_0(\Mod \Lambda) \otimes_{\mathbb Z} {\mathbb R} \rightarrow {\mathbb R}
\]
be an ${\mathbb R}$-linear form, where $a_{U_1}>0$ and $a_{U_2}>0$. We can calculate the values of $\theta$ for indecomposable $\Lambda$-modules as follows:
\[
\xymatrix@!C=2mm@!R=0.5mm{
 -a_{U_2} \ar[rd] && -a_{U_2} \ar[rd] && -a_{U_2} \ar[rd] && -a_{U_2}\\
 & -a_{U_1}-a_{U_2} \ar[ru] \ar[rd] && 0 \ar[ru] \ar[rd] && a_{U_1}-a_{U_2} \ar[ru] \ar[rd] &\\
 -a_{U_2} \ar[ru] && -a_{U_1} \ar[ru] && a_{U_1} \ar[ru] && -a_{U_2}}
\]
 Then the $\theta$-semistable subcategory of $\Mod \Lambda$ is $\add \Bigl(\hspace{-1mm}\mbox{\footnotesize$\renewcommand{\arraystretch}{0.6}\begin{array}{c} 2\\ 3\\ \end{array}$}\hspace{-1mm}\Bigr)$. Thus $\mathcal{W}_U=\add \Bigl(\hspace{-1mm}\mbox{\footnotesize$\renewcommand{\arraystretch}{0.6}\begin{array}{c} 2\\ 3\\ \end{array}$}\hspace{-1mm}\Bigr)$ is wide and semistable.


 (2) Let $T=U \oplus U_3 \in {\rm 2 \mathchar`-silt}\Lambda$, where
\[
   U_3=(P_1\rightarrow P_2).
\]
 Then there is a triangle
\[
 \Lambda \rightarrow U_3 \oplus U_3 \rightarrow U_1 \oplus U_2 \oplus U_2 \oplus U_2 \rightarrow \Lambda[1]
\]
 in ${\mathsf K}^{{\rm b}}(\proj \Lambda)$. Thus $T_{\lambda}=U_3$ and $T_{\rho}=U_1 \oplus U_2=U$. By Theorem \ref{main}, $\mathcal{W}^{T}$ is the $\theta$-semistable subcategory of $\Mod \Lambda$ for the above $\theta$. In particular, we have $\mathcal{W}^{T}=\mathcal{W}_{U}=\add \Bigl(\hspace{-1mm}\mbox{\footnotesize$\renewcommand{\arraystretch}{0.6}\begin{array}{c} 2\\ 3\\ \end{array}$}\hspace{-1mm}\Bigr)$, as the second row of the right column in Table \ref{taubijex} shows.


\section{Proofs of our results}


\subsection{Preliminary}\label{prelim}
 We recall unexplained terminologies and the bijections of Theorem \ref{MSresult} from \cite{Ai,AI,AIR,ASS,KV}.

 Let ${\mathcal S}$ be a full subcategory of $\Mod \Lambda$. We call ${\mathcal S}$ a {\it torsion class} (resp., {\it torsion free class}) if it is closed under extensions and quotients (resp., extensions and submodules) \cite{ASS}. For subcategories ${\mathcal T}$ and ${\mathcal F}$ of $\Mod \Lambda$, a pair $({\mathcal T},{\mathcal F})$ is called a {\it torsion pair} if ${\mathcal T}={}^{\perp}{\mathcal F}$ and ${\mathcal F}={\mathcal T}^{\perp}$. Then ${\mathcal T}$ is a torsion class and ${\mathcal F}$ is a torsion free class. Conversely, any torsion class (resp., torsion free class) gives rise to a torsion pair. We call ${\mathcal S}$ {\it functorially finite} if any $\Lambda$-module admits both a left and a right ${\mathcal S}$-approximation. More precisely, for any $M \in \Mod \Lambda$, there are morphisms $g_1 : M \rightarrow S_1$ and $g_2 : S_2 \rightarrow M$ with $S_1, S_2 \in {\mathcal S}$ such that ${\rm Hom}_{\Lambda}(g_1,S)$ and ${\rm Hom}_{\Lambda}(S,g_2)$ are surjective for any $S \in {\mathcal S}$. Then $g_1$ is called a left ${\mathcal S}$-approximation of $M$ and $g_2$ is called a right ${\mathcal S}$-approximation of $M$. We call ${\mathcal S}$ {\it left finite} if the minimal torsion class containing ${\mathcal S}$ is functorially finite (see \cite{As}).

 Let $T \in \Mod \Lambda$. We call $T$ {\it $\tau$-rigid} if ${\rm Hom}_{\Lambda}(T,\tau T)=0$, where $\tau$ is the Auslander-Reiten translation of $\Mod \Lambda$ \cite{AIR}. We call $T$ {\it support $\tau$-tilting} if $T$ is $\tau$-rigid and $|T|=|\Lambda/\langle e \rangle|$ for some idempotent $e$ of $\Lambda$ such that $e T=0$, where $|T|$ is the number of non-isomorphic indecomposable direct summands of $T$.

 Let $P \in {\mathsf K}^{{\rm b}}(\proj \Lambda)$. We call $P$ {\it presilting} if ${\rm Hom}_{{\mathsf K}^{{\rm b}}(\proj \Lambda)}(P,P[i])=0$ for any $i>0$ \cite{Ai,AI,KV}. We call $P$ {\it silting} if $P$ is presilting and satisfies $\thick P = {\mathsf K}^{{\rm b}}(\proj \Lambda)$, where $\thick P$ is the smallest subcategory of ${\mathsf K}^{{\rm b}}(\proj \Lambda)$ containing $P$ which is closed under shifts, cones and direct summands. We say that $P=(P^i,d^i)$ is {\it two-term} if $P^i=0$ for all $i \neq 0,-1$. We denote by ${\rm 2 \mathchar`-presilt}\Lambda$ (resp., ${\rm 2 \mathchar`-silt}\Lambda$) the set of isomorphism classes of basic two-term presilting (resp., silting) complexes in ${\mathsf K}^{{\rm b}}(\proj \Lambda)$.

 The bijections of Theorem \ref{MSresult} are given in the following way \cite[Theorem 2.7, 3.2]{AIR}\cite[Theorem 3.10]{MS}\cite[Theorem]{S}:
  \begin{eqnarray*}
    (1')\rightarrow(1) &:& T \mapsto {\rm H}^0(T),\\
    (1')\rightarrow(2) &:& T \mapsto \mathcal{T}_T^{-}=\Fac {\rm H}^0(T), \hspace{18mm} (1')\rightarrow(2') \hspace{2mm} : \hspace{2mm} T \mapsto \mathcal{F}_T^{+}=\Sub {\rm H}^{-1}(\nu T),\\
    (2)\rightarrow(2') &:& {\mathcal T} \mapsto {\mathcal T}^{\perp}, \hspace{39mm} (2)\leftarrow(2') \hspace{2mm} : \hspace{2mm} {}^{\perp}{\mathcal F} \hspace{0.7mm} \mbox{\reflectbox{$\mapsto$}} \hspace{0.7mm} {\mathcal F},\\
    (1')\rightarrow(3) &:& T \mapsto \mathcal{W}^{T}=\Fac {\rm H}^0(T) \cap {\rm H}^0(T_{\rho})^{\perp},
  \end{eqnarray*}
where ${\rm H}^i(T)$ is the $i$-th cohomology of $T$. Recall that we have
\begin{equation}\label{silttor}
 ({}^{\perp}{\rm H}^{-1}(\nu T),\Sub {\rm H}^{-1}(\nu T))=(\mathcal{T}_T^{+},\mathcal{F}_T^{+})=(\mathcal{T}_T^{-},\mathcal{F}_T^{-})=(\Fac {\rm H}^0(T),{\rm H}^0(T)^{\perp})
\end{equation}
for $T \in {\rm 2 \mathchar`-silt}\Lambda$ \cite[Proposition 2.16]{AIR}.


\subsection{Linear forms on Grothendieck groups}

 Let $\Lambda$ be a finite dimensional algebra over a field $k$. Let $K_0(\Mod \Lambda)$ and $K_0(\proj \Lambda)$ be the Grothendieck groups of the abelian category $\Mod \Lambda$ and the exact category $\proj \Lambda$ with only split short exact sequences, respectively. Then we have natural isomorphisms $K_0(\Mod \Lambda) \simeq K_0({\mathsf D}^{{\rm b}}(\Mod \Lambda))$ and $K_0(\proj \Lambda) \simeq K_0({\mathsf K}^{{\rm b}}(\proj \Lambda))$. Moreover, $K_0(\Mod \Lambda)$ has a basis consisting of the isomorphism classes $S_i$ of simple $\Lambda$-modules, and $K_0(\proj \Lambda)$ has a basis consisting of the isomorphism classes $P_i$ of indecomposable projective $\Lambda$-modules, where ${\rm top}P_i=S_i$.

 The Euler form is a non-degenerate pairing between $K_0(\proj \Lambda)$ and $K_0(\Mod \Lambda)$ given by
\begin{equation}\label{Euler}
\langle P,M \rangle := \sum_{i \in {\mathbb Z}}(-1)^i{\rm dim}_k{\rm Hom}_{{\mathsf D}^{{\rm b}}(\Mod \Lambda)}(P,M[i])
\end{equation}
 for any $P \in {\mathsf K}^{{\rm b}}(\proj \Lambda)$ and $M \in {\mathsf D}^{{\rm b}}(\Mod \Lambda)$. Then $\{P_i\}$ and $\{S_i\}$ are dual bases of each other. In particular, we have a ${\mathbb Z}$-linear form $\langle P,-\rangle : K_0(\Mod \Lambda) \rightarrow {\mathbb Z}$ for $P \in {\mathsf K}^{{\rm b}}(\proj \Lambda)$.

 Recall that there is a Serre duality, that is, a bifunctorial isomorphism
\begin{equation}\label{nakayama}
 {\rm Hom}_{{\mathsf D}^{{\rm b}}(\Mod \Lambda)}(P,M) \simeq D{\rm Hom}_{{\mathsf D}^{{\rm b}}(\Mod \Lambda)}(M,\nu P)
\end{equation}
for $P \in {\mathsf K}^{{\rm b}}(\proj \Lambda)$ and $M \in {\mathsf D}^{{\rm b}}(\Mod \Lambda)$. The following observation is basic.

\begin{lemma}\label{eulerform}
 Let $P$ be a two-term complex in ${\mathsf K}^{{\rm b}}(\proj \Lambda)$. For $M \in \Mod \Lambda$, we have
{\setlength\arraycolsep{0.5mm}
\begin{eqnarray}
 \langle P,M\rangle &=& {\rm dim}_{k}{\rm Hom}_{\Lambda}({\rm H}^{0}(P),M)-{\rm dim}_{k}{\rm Hom}_{{\mathsf D}^{{\rm b}}(\Mod \Lambda)}(P,M[1]) \label{eulerform1}\\
 &=& {\rm dim}_{k}{\rm Hom}_{{\mathsf D}^{{\rm b}}(\Mod \Lambda)}(M,\nu P)-{\rm dim}_{k}{\rm Hom}_{\Lambda}(M,{\rm H}^{-1}(\nu P)). \label{eulerform2}
\end{eqnarray}}
\end{lemma}

\begin{proof}
 Since $P$ is two-term, ${\rm Hom}_{{\mathsf D}^{{\rm b}}(\Mod \Lambda)}(P,M[i])=0$ holds for any $i \neq 0, 1$. Moreover, we have ${\rm Hom}_{{\mathsf D}^{{\rm b}}(\Mod \Lambda)}(P,M)={\rm Hom}_{\Lambda}({\rm H}^{0}(P),M)$. Thus (\ref{eulerform1}) holds. Similarly, (\ref{eulerform2}) follows from (\ref{nakayama}).
\end{proof}


\subsection{Proofs of Theorems \ref{main} and \ref{presiltsemistable}}

 First, we make preparations to prove Theorem \ref{presiltsemistable}.

\begin{lemma}\label{torsilt}
 For $U \in {\rm 2 \mathchar`-presilt}\Lambda$, we have
{\setlength\arraycolsep{0.5mm}
\begin{eqnarray}
 \mathcal{T}_U^{+}&=&\{M \in \Mod \Lambda \mid {\rm Hom}_{{\mathsf D}^{{\rm b}}(\Mod \Lambda)}(U,M[1])=0\},\label{torsil}\\
 \mathcal{F}_U^{-}&=&\{M \in \Mod \Lambda \mid {\rm Hom}_{{\mathsf D}^{{\rm b}}(\Mod \Lambda)}(M,\nu U)=0\}.\label{torfsil}
\end{eqnarray}}
\end{lemma}

\begin{proof}
 For $M \in \Mod \Lambda$, by (\ref{nakayama}) we have
\[
 D{\rm Hom}_{{\mathsf D}^{{\rm b}}(\Mod \Lambda)}(U,M[1]) \simeq {\rm Hom}_{{\mathsf D}^{{\rm b}}(\Mod \Lambda)}(M[1],\nu U) \simeq {\rm Hom}_{\Lambda}(M,{\rm H}^{-1}(\nu U)).
\]
 Thus (\ref{torsil}) holds. Similarly, (\ref{torfsil}) holds.
\end{proof}

 Using ${\mathbb R}$-linear forms on $K_0(\Mod \Lambda) \otimes_{\mathbb Z} {\mathbb R}$, we have the following properties of torsion pairs $(\mathcal{T}_U^{+},\mathcal{F}_U^{+})$ and $(\mathcal{T}_U^{-},\mathcal{F}_U^{-})$ for $U \in {\rm 2 \mathchar`-presilt}\Lambda$.

\begin{proposition}\label{thetavalue}
 Let $U \in {\rm 2 \mathchar`-presilt}\Lambda$ and $\theta$ the corresponding ${\mathbb R}$-linear form on $K_0(\Mod \Lambda) \otimes_{\mathbb Z} {\mathbb R}$ defined in Theorem \ref{presiltsemistable}. For $M \in \Mod \Lambda$, the following assertions hold.\par
 (a) If $M \in \mathcal{T}_U^{+}$, then $\theta(M) \ge 0$. Moreover, if $M \in \mathcal{T}_U^{-}$ is non-zero, then $\theta(M) > 0$.\par
 (b) If $M \in \mathcal{F}_U^{-}$, then $\theta(M) \le 0$. Moreover, if $M \in \mathcal{F}_U^{+}$ is non-zero, then $\theta(M) < 0$.
\end{proposition}

\begin{proof}
 Let $M \in \mathcal{T}_U^{+}$. By (\ref{eulerform1}) and (\ref{torsil}), we have
\[
\theta(M) = \sum_{X} a_{X}{\rm dim}_{k}{\rm Hom}_{\Lambda}({\rm H}^{0}(X),M),
\]
where $X$ runs over all indecomposable direct summands of $U$. Since $a_X>0$ holds for any $X$, (a) holds. Similarly, (b) holds by (\ref{eulerform2}) and (\ref{torfsil}).
\end{proof}


Now, we are ready to prove Theorem \ref{presiltsemistable}.

\begin{proof}[Proof of Theorem \ref{presiltsemistable}]
 Let $M \in \mathcal{W}_U=\mathcal{T}_U^{+} \cap \mathcal{F}_U^{-}$. Then $\theta(M)=0$ holds by Proposition \ref{thetavalue}. Since $\mathcal{F}_U^{-}$ is a torsion free class, any submodule $L$ of $M$ is also belongs to $\mathcal{F}_U^{-}$. Thus $\theta(L) \le 0$ holds by Proposition \ref{thetavalue}(b). Therefore, $M$ is $\theta$-semistable.

 Conversely, assume that $M \in \Mod \Lambda$ is $\theta$-semistable. Since $(\mathcal{T}_U^{-},\mathcal{F}_U^{-})$ is a torsion pair, there is an exact sequence
\[
0 \rightarrow L \rightarrow M \rightarrow N \rightarrow 0,
\]
where $L \in \mathcal{T}_U^{-}$ and $N \in \mathcal{F}_U^{-}$. Since $\theta(L)\le0$ holds, we have $L=0$ by Proposition \ref{thetavalue}(a). Thus $M=N \in \mathcal{F}_U^{-}$. Similarly, taking a canonical sequence of $M$ with respect to the torsion pair $(\mathcal{T}_U^{+},\mathcal{F}_U^{+})$, we have $M \in \mathcal{T}_U^{+}$. Thus $M \in \mathcal{T}_U^{+} \cap \mathcal{F}_U^{-} = \mathcal{W}_U$ holds.
\end{proof}


 Next, we make preparations to prove Theorem \ref{main}. For $T \in {\rm 2 \mathchar`-silt}\Lambda$, we have the following characterization of the corresponding torsion pairs $(\mathcal{T}_T^{+},\mathcal{F}_T^{+})$ and $(\mathcal{T}_T^{-},\mathcal{F}_T^{-})$ in $\Mod \Lambda$.

\begin{lemma}\label{trivial}
 Let $T=T_{\lambda} \oplus T_{\rho} \in {\rm 2 \mathchar`-silt}\Lambda$ as in (\ref{triangle}). The following equalities hold.
\begin{enumerate}
 \item[(a)] $\mathcal{T}_T^{+}=\mathcal{T}_T^{-}=\Fac {\rm H}^{0}(T_{\lambda})={}^{\perp}{\rm H}^{-1}(\nu T_{\rho})$.
 \item[(b)] $\mathcal{F}_T^{+}=\mathcal{F}_T^{-}={\rm H}^0(T_{\lambda})^{\perp}=\Sub {\rm H}^{-1}(\nu T_{\rho})$.
\end{enumerate}
\end{lemma}

\begin{proof}
 By (\ref{silttor}), we have $\mathcal{T}_T^{+}=\mathcal{T}_T^{-}$ and $\mathcal{F}_T^{+}=\mathcal{F}_T^{-}$.

 Applying ${\rm H}^{0}(-)$ to the triangle $\Lambda \rightarrow T' \rightarrow T'' \rightarrow \Lambda[1]$ in (\ref{triangle}), we have an exact sequence
\[
 \Lambda \rightarrow {\rm H}^{0}(T') \rightarrow {\rm H}^{0}(T'') \rightarrow 0
\]
in $\Mod \Lambda$. Thus ${\rm H}^{0}(T_{\rho}) \in \Fac {\rm H}^{0}(T_{\lambda})$ holds. Hence we have $\mathcal{T}_T^{-}=\Fac {\rm H}^{0}(T)=\Fac {\rm H}^{0}(T_{\lambda})$ and $\mathcal{F}_T^{-}={\rm H}^0(T)^{\perp}={\rm H}^0(T_{\lambda})^{\perp}$. Dually, the equations $\mathcal{T}_T^{+}={}^{\perp}{\rm H}^{-1}(\nu T_{\rho})$ and $\mathcal{F}_T^{+}=\Sub {\rm H}^{-1}(\nu T_{\rho})$ hold.
\end{proof}

 The following observation gives a connection between two constructions $\mathcal{W}^{(-)}$ and $\mathcal{W}_{(-)}$ of wide subcategories.


\begin{lemma}\label{2wide}
 Let $T=T_{\lambda} \oplus T_{\rho} \in {\rm 2 \mathchar`-silt}\Lambda$. Then $\mathcal{W}^{T}=\mathcal{W}_{T_{\rho}}$ holds.
\end{lemma}

\begin{proof}
 There are equalities
\[
\mathcal{W}^{T}\overset{(\ref{W^T})}{\scalebox{3}[1]{=}}\mathcal{T}_T^{-} \cap {\rm H}^0(T_{\rho})^{\perp}\overset{\mbox{{\tiny Lemma }} \ref{trivial}(a)}{\scalebox{8}[1]{=}}{}^{\perp}{\rm H}^{-1}(\nu T_{\rho}) \cap {\rm H}^0(T_{\rho})^{\perp}=\mathcal{T}_{T_{\rho}}^{+} \cap \mathcal{F}_{T_{\rho}}^{-}=\mathcal{W}_{T_{\rho}}. \qedhere
\]
\end{proof}


 This result enables us to prove Theorem \ref{main}.

\begin{proof}[Proof of Theorem \ref{main}]
 The assertion immediately follows from Lemma \ref{2wide} and Theorem \ref{presiltsemistable}.
\end{proof}

\medskip\noindent{\bf Acknowledgements}.
The author is a Research Fellow of Society for the Promotion of Science (JSPS). This work was supported by JSPS KAKENHI Grant Number JP17J04270.\par
The author would like to thank his supervisor Osamu Iyama for his guidance and helpful advice. He is also grateful to Laurent Demonet for helpful comments.

\end{document}